\documentclass[reqno, a4paper]{amsart}

\usepackage{fixltx2e}
\usepackage[T1]{fontenc}
\usepackage[british]{babel}
\usepackage[pdftex]{graphicx}
\usepackage[pdftex]{color}
\usepackage{amsfonts, amssymb, amsthm, amsmath}
\usepackage[all]{xy}
\usepackage{hyperref}
\usepackage{mathbbol}
\usepackage{calrsfs}
\usepackage{mathabx}
\usepackage{wasysym}

\theoremstyle{plain}
\newtheorem{lemma}[subsection]{Lemma}
\newtheorem{theorem}[subsection]{Theorem}
\newtheorem{proposition}[subsection]{Proposition}
\newtheorem{corollary}[subsection]{Corollary}

\theoremstyle{definition}
\newtheorem{example}[subsection]{Example}
\newtheorem{definition}[subsection]{Definition}

\newtheorem{remark}[subsection]{Remark}

\newcommand{\comp}{\raisebox{0.2mm}{\ensuremath{\scriptstyle{\circ}}}}
\newcommand{\defn}{\textbf}
\renewcommand{\equiv}{$\Leftrightarrow$}
\renewcommand{\implies}{$\Rightarrow$}
\newcommand{\join}{\vee}
\newcommand{\meet}{\wedge}

\newcommand{\musgib}[2]{\left\langle\begin{smallmatrix} {#1} \\ {#2}\end{smallmatrix}\right\rangle}
\newcommand{\muss}[3]{\left\langle\begin{smallmatrix} {#1} \\ {#2} \\ {#3}\end{smallmatrix}\right\rangle}
\newcommand{\noproof}{\hfill\qed}
\newcommand{\normal}{\ensuremath{\lhd}}
\newcommand{\protosplit}{\ensuremath{\LHD}}
\newcommand{\tensor}{\diamond}
\DeclareMathOperator{\ab}{Ab}
\DeclareMathOperator{\coker}{coker}
\DeclareMathOperator{\Eq}{Eq}
\DeclareMathOperator{\Ker}{Ker}

\newcommand{\X}{\ensuremath{\mathcal{X}}}

\newcommand{\KK}{\ensuremath{\mathbb{K}}}
\newcommand{\ZZ}{\ensuremath{\mathbb{Z}}}

\newcommand{\Ab}{\ensuremath{\mathsf{Ab}}}
\newcommand{\Pt}{\ensuremath{\mathsf{Pt}}}

\renewcommand{\H}{\ensuremath{\mathrm{H}}}
\renewcommand{\S}{\ensuremath{\mathrm{S}}}
\newcommand{\Z}{\ensuremath{\mathrm{Z}}}

\newcommand{\NH}{{\rm (NH)}}
\newcommand{\PA}{{\rm (PA)}}
\newcommand{\SH}{{\rm (SH)}}
\newcommand{\UCE}{{\rm (UCE)}}
\newcommand{\WNH}{{\rm (WNH)}}

\def\pullback{
 \ar@{-}[]+R+<6pt,-1pt>;[]+RD+<6pt,-6pt>%
 \ar@{-}[]+D+<1pt,-6pt>;[]+RD+<6pt,-6pt>}
 
\def\pushout{%
 \ar@{-}[]+L+<-6pt,1pt>;[]+LU+<-6pt,6pt>%
 \ar@{-}[]+U+<-1pt,6pt>;[]+LU+<-6pt,6pt>}

\def\splitpullback{%
 \ar@{-}[]+R+<6pt,-.51ex>;[]+RD+<6pt,-6pt>%
 \ar@{-}[]+D+<.51ex,-6pt>;[]+RD+<6pt,-6pt>}

\def\skewpullback{%
 \ar@{-}[]+LD+<-6pt,-6pt>;[]+LDD+<-6pt,-15.5pt>%
 \ar@{-}[]+D+<-1pt,-6pt>;[]+LDD+<-6pt,-15.5pt>}
 
 \def\sqpullback{%
 \ar@{-}[]+RU+<5pt,3pt>;[]+RD+<5pt,-5pt>%
 \ar@{-}[]+D+<-4pt,-5pt>;[]+RD+<5pt,-5pt>}
 
\newdir{>>}{{}*!/3.5pt/:(1,-.2)@^{>}*!/3.5pt/:(1,+.2)@_{>}*!/7pt/:(1,-.2)@^{>}*!/7pt/:(1,+.2)@_{>}}
\newdir{ >>}{{}*!/8pt/@{|}*!/3.5pt/:(1,-.2)@^{>}*!/3.5pt/:(1,+.2)@_{>}}
\newdir{ |>}{{}*!/-3.5pt/@{|}*!/-8pt/:(1,-.2)@^{>}*!/-8pt/:(1,+.2)@_{>}}
\newdir{ >}{{}*!/-8pt/@{>}}
\newdir{>}{{}*:(1,-.2)@^{>}*:(1,+.2)@_{>}}
\newdir{<}{{}*:(1,+.2)@^{<}*:(1,-.2)@_{<}}

\begin{document}

\title[Peri-abelian categories and (UCE)]{Peri-abelian categories and\\ the universal central extension condition}

\author{James R.~A.~Gray}
\author{Tim Van~der Linden}

\email{jamesgray@sun.ac.za}
\email{tim.vanderlinden@uclouvain.be}

\address{Mathematics Division, Department of Mathematical Sciences, Stellenbosch University, Private Bag X1, Matieland 7602, South Africa}
\address{Institut de recherche en math\'ematique et physique, Universit\'e catholique de Louvain, chemin du cyclotron~2 bte~L7.01.02, 1348 Louvain-la-Neuve, Belgium}

\thanks{Tim Van der Linden is a Research Associate of the Fonds de la Recherche Scientifique--FNRS. He wishes to thank UNISA for its kind hospitality during his stay in Johannesburg}



\begin{abstract}
We study the relation between Bourn's notion of peri-abelian category and conditions involving the coincidence of the Smith, Huq and Higgins commutators. In particular we show that a semi-abelian category is peri-abelian if and only if for each normal subobject $K\normal X$, the Higgins commutator of $K$ with itself coincides with the normalisation of the Smith commutator of the denormalisation of $K$ with itself. We show that if a category is peri-abelian, then the condition (UCE), which was introduced and studied by Casas and the second author, holds for that category. In addition we show, using amongst other things a result by Cigoli, that all categories of interest in the sense of Orzech are peri-abelian and therefore satisfy the condition (UCE).
\end{abstract}

\maketitle

\section*{Introduction}
Using Janelidze and Kelly's general notion of central extension~\cite{Janelidze-Kelly}, the classical theory of universal central extensions valid for groups and Lie algebras---see, for instance,~\cite{Milnor, Weibel}---may be generalised to the context of semi-abelian categories~\cite{Janelidze-Marki-Tholen, Borceux-Bourn} with enough projectives. As explained in~\cite{CVdL}, \emph{most} of this generalisation is entirely straightforward. Somewhat surprisingly though, there is a difficulty in obtaining a general version of the standard recognition theorem for universal central extensions, which characterises universality of a central extension in terms of properties of its domain. In the case of groups, this result says that a central extension $u\colon {U\to Y}$ of groups is universal if and only if $\H_{1}(U,\ZZ)=\H_{2}(U,\ZZ)=0$ or, equivalently, if and only if $U$ is perfect and every central extension of the group $U$ splits~\cite{Milnor}.

As it turns out, this general theory of universal central extensions works well when the underlying semi-abelian category satisfies an additional requirement, called the \defn{universal central extension condition} or \defn{(UCE)} in~\cite{CVdL}, that is, if $B$ is a perfect object and $f\colon{A\to B}$ and $g\colon{B\to C}$ are central extensions in $\X$, then the extension $g\comp f$ is also central. Indeed, for any perfect object $U$ of $\X$, the statements
\begin{enumerate}
\item each central extension $u\colon {U\to Y}$ is universal;
\item each central extensions of $U$ splits;
\item each universal central extensions of $U$ splits;
\item $\H_{2}(U)=0$
\end{enumerate}
are equivalent if and only if \UCE\ holds. Here we assume that $\X$ has enough projectives; furthermore, centrality, perfectness and homology are all defined with respect to the Birkhoff subcategory $\Ab(\X)$ of abelian objects of $\X$. The condition \UCE\ clearly holds for groups and Lie algebras; on the other hand, the category of non-associative algebras over a field is semi-abelian but does not satisfy \UCE, which shows that this condition does not hold in an arbitrary semi-abelian category.
The aim of the present paper is to understand how the condition \UCE\ is related to other conditions occurring in categorical algebra, in particular which conditions it follows from. We analyse it in terms of basic commutator conditions, proving that it is closely related to the notion of \emph{peri-abelian} category with appeared in recent work by Bourn~\cite{Bourn-Peri}. For any semi-abelian category there is a natural notion of action~\cite{Bourn-Janelidze:Semidirect}---generalising that of a group $G$ acting on a group---as well as Beck's notion of $G$-module~\cite{Beck}. The category of groups being peri-abelian amounts to the fact that the universal way to make a $G$-action on a group $X$ into a $G$-module is to abelianise $X$. For semi-abelian categories this becomes a condition which may or may not hold. We show in Proposition~\ref{(SH) + (WNH) => (PA)} that a semi-abelian category is peri-abelian (satisfies \defn{condition~(PA)}) if and only if there is a partial coincidence of the Higgins and Smith commutators in it: the Higgins~\cite{MM-NC} commutator $[K,K]$ of any normal subobject $K\normal X$ in it is the normalisation (= zero-class, see~\cite{Bourn2000, Borceux-Bourn}) of the Smith~\cite{Smith, Pedicchio} commutator $[R,R]^{\S}$, where $R$ is the equivalence relation corresponding to $K$ (=~its \emph{denormalisation}). As a consequence, combining results in~\cite{AlanThesis}, \cite{Montoli} and~\cite{BJ07}, we see that any \emph{category of interest} in the sense of Orzech~\cite{Orzech} is peri-abelian. It is not a coincidence that categories of non-associative rings need not be such~\cite[Example~5.3.7]{AlanThesis}. Indeed---this is Theorem~\ref{Theorem (NH) => (UCE)}---a semi-abelian category which is peri-abelian will always satisfy~\UCE, as explained in Section~\ref{Section (UCE)}.

We start with a revision of some basic commutator theory in Section~\ref{Section Preliminaries}. In Section~\ref{Section (WNH) and (PA)} we reformulate the concept of peri-abelian category in the language of commutators, which gives us the equivalent conditions of Proposition~\ref{(SH) + (WNH) => (PA)}. The final Section~\ref{Section (UCE)} leads towards our main Theorem~\ref{Theorem (NH) => (UCE)}: \PA\ implies \UCE.

\section{Preliminaries}\label{Section Preliminaries}
Throughout the text we assume that $\X$ is a semi-abelian category. In this section we recall the definitions and some basic properties of the Huq, Smith and Higgins commutators. Before doing so let us introduce some terminology and notation. We will call a diagram 
\[
\xymatrix{
X \ar@<.5ex>@{-{ >>}}[r]^-{f} & Y \ar@<.5ex>[l]^-{s}
}
\]
where $fs=1_Y$ a \defn{point in $\X$}, and a diagram
\[
\xymatrix{X' \ar[d]_-{\theta} \ar@{-{ >>}}@<.5ex>[r]^{f'} & Y' \ar[d]^{\phi} \ar@<.5ex>[l]^{s'} \\
X \ar@<.5ex>@{-{ >>}}[r]^-{f} & Y \ar@<.5ex>[l]^-{s}}
\]
where the top and bottom rows are points and $\phi f' = f\theta$ and $\theta s' = s\phi$, a \defn{morphism of points}. We will denote by $\Pt(\X)$ the category of points and by $\Pt_Y(\X)$ the fiber above $Y$ of the fibration sending a point to the codomain of its split epimorphism. We will call a diagram
\[
\xymatrix{
K \ar[r]^{k} & X \ar@<.5ex>@{-{ >>}}[r]^-{f} & Y \ar@<.5ex>[l]^-{s}
}
\]
where $fs=1_Y$ and $k$ is a kernel of $f$ a \defn{point with chosen kernel}, and a diagram 
\[
\xymatrix{K' \ar[d]_{\kappa}\ar[r]^{k'} & X' \ar[d]^{\theta} \ar@{-{ >>}}@<.5ex>[r]^{f'} & Y' \ar[d]^{\phi} \ar@<.5ex>[l]^{s'} \\
K \ar[r]_{k} &X \ar@<.5ex>@{-{ >>}}[r]^-{f} & Y \ar@<.5ex>[l]^-{s}}
\]
where the top and bottom rows are points with chosen kernels and $\theta k'=k\kappa$, $\phi f' = f\theta$ and $\theta s' = s\phi$, a \defn{morphism of points with chosen kernel}.

\subsection{The Huq commutator}\label{Huq commutator}
A cospan of monomorphisms in $\X$ as on left
\begin{equation*}\label{Cospan}
\vcenter{\xymatrix{K \ar@{{ >}->}[r]^-{k} & X & L \ar@{{ >}->}[l]_-{l}}}
\qquad\qquad
\vcenter{\xymatrix@!0@=3em{ & K \ar[ld]_{\langle 1_{K},\,0\rangle} \ar[rd]^-{k} \\
K\times L \ar@{.>}[rr]|-{\varphi} && X\\
& L \ar[lu]^{\langle 0,\,1_{L}\rangle} \ar[ru]_-{l}}}
\end{equation*}
is said to \defn{(Huq-)commute}~\cite{BG,Huq} when there exists a (necessarily unique) morphism~$\varphi$ making the diagram on the right commute. The \defn{Huq commutator of~$k$ and~$l$}~\cite{Bourn-Gran, Borceux-Bourn} is defined to be the smallest normal subobject $[k,l]_{X}\colon {[K,L]_{X}\to X}$, making the images of $k$ and~$l$ commute in the quotient $X/[K,L]$. In this context it can be shown that the Huq commutator always exists, and can be constructed as as the kernel $[K,L]_{X}$ of the (normal epi)morphism ${X\to Q}$, where $Q$ is the colimit of the outer square above.

\subsection{The Smith commutator}\label{Smith commutator}
Given a pair of equivalence relations $(R,S)$ on a common object $X$ of $\X$ as on the left
\begin{equation*}\label{Category-RG}
\vcenter{\xymatrix@!0@=5em{R \ar@<1.5ex>[r]^-{r_{1}} \ar@<-1.5ex>[r]_-{r_{2}} & X \ar[l]|-{\Delta_R} \ar[r]|-{\Delta_S} & S \ar@<1.5ex>[l]^-{s_{1}} \ar@<-1.5ex>[l]_-{s_{2}}}}
\qquad\quad
\vcenter{\xymatrix@!0@=4em{R\times_{X}S \pullback \ar[r]^-{\pi_{S}} \ar[d]_-{\pi_{R}} & S \ar[d]^-{s_{1}} \\
R \ar[r]_-{r_{2}} & X}}
\qquad\quad
\vcenter{\xymatrix@!0@=3em{ & R \ar[ld]_{\langle 1_{R},\,\Delta_S \circ r_{2}\rangle} \ar[rd]^-{r_{1}} \\
R\times_{X}S \ar@{.>}[rr]|-{\theta} && X\\
& S \ar[lu]^{\langle\Delta_R\circ s_{1},\,1_{S}\rangle} \ar[ru]_-{s_{2}}}
}
\end{equation*}
consider the induced pullback of $r_{2}$ and $s_{1}$ in the middle. The equivalence relations $R$ and $S$ are said to \defn{centralise each other} or to \defn{(Smith-)commute}~\cite{Smith, Pedicchio, BG} when there exists a (necessarily unique) morphism $\theta$ making the diagram on the right commute. In a similar way as for the Huq commutator, the \defn{Smith commutator} is defined to be the smallest equivalence relation $[R,S]^{\S}$ on $X$, making the images of~$R$ and~$S$ in the quotient $X/[R,S]^{\S}$ commute. In this context it can be shown to always exist, since it admits a construction similar to the Huq commutator's. It follows that $R$ and $S$ commute if and only if~${[R,S]^{\S}=\Delta_{X}}$, where~$\Delta_{X}$ denotes the smallest equivalence relation on $X$.

We say that $R$ is a \defn{central} equivalence relation when it commutes with $\nabla_X$, the largest equivalence relation on $X$, so that~$[R,\nabla_X]^{\S}=\Delta_X$. A \defn{central extension} is a regular epimorphism $f\colon{X\to Y}$ whose kernel pair $\Eq(f)$ is a central equivalence relation.

Smith commutators characterise internal groupoids~\cite{Pedicchio}: a reflexive graph
\[
\xymatrix{X \ar@<1ex>[r]^-{d} \ar@<-1ex>[r]_-{c} & Y \ar[l]|-{e}}
\]
in $\X$ will be a groupoid if and only if $[\Eq(d),\Eq(c)]^{\S}=\Delta_{X}$. In particular, it characterises Beck modules~\cite{Beck}, since any Beck module in $\X$, which is an abelian object $(f,s)\colon X\rightleftarrows Y$ in the category $\Pt_{Y}(\X)$ of points over~$Y$, so in particular a split extension $f$ with chosen splitting $s$, may be seen as an internal groupoid of the form
\[
\xymatrix{X \ar@<1ex>[r]^-{f} \ar@<-1ex>[r]_-{f} & Y; \ar[l]|-{s}}
\]
see~\cite{Bourn-Janelidze:Torsors} where this is explained in detail. Hence the abelianisation of a point $(f,s)\colon X\rightleftarrows Y$ is obtained through the quotient $X/[\Eq(f),\Eq(f)]^{\S}$.

\subsection{The coincidence of the Smith and Huq commutators}\label{(SH)}
It is well known, and easily verified, that if the Smith commutator $[R,S]^{\S}$ of two equivalence relations $R$ and $S$ is trivial, then the Huq commutator $[K,L]_{X}$ of their normalisations $K$ and $L$ is also trivial~\cite{BG}. It is also well known that, in general, the converse is false; there are counterexamples in the category of digroups~\cite{Borceux-Bourn, Bourn2004}, which is a variety of~$\Omega$-groups~\cite{Higgins} (and hence semi-abelian), and in the semi-abelian variety of loops~\cite{HVdL}. The requirement that the two commutators vanish together is known as the condition \defn{(SH)}. As explained in~\cite{MFVdL, HVdL}, it is important in the study of internal crossed modules~\cite{Janelidze}. The condition \SH\ holds for all action accessible categories~\cite{BJ07}, hence, in particular~\cite{Montoli}, for any \emph{category of interest} in the sense of Orzech~\cite{Orzech}.

In order to simplify our notations, we shall write $[K,L]^{\S}$ for the normalisation of the Smith commutator $[R,S]^{\S}$ (which coincides with the \emph{Ursini commutator of $K$ and~$L$} defined and studied in~\cite{Mantovani:Ursini}). The condition \SH\ for a semi-abelian category~$\X$ then amounts to the equality $[K,L]^{\S}=[K,L]_{X}$ for all $K$, $L\normal X$ in~$\X$.

The special case of central extensions is worth mentioning. A regular epimorphism $f\colon{X\to Y}$ with kernel $K$ is central in the above sense if and only if $[K,X]^{\S}=0$. It was shown in~\cite{Gran-VdL} that always $[K,X]^{\S}=[K,X]_{X}$, so centrality of~$f$ may be expressed as the vanishing of a Huq commutator, that is, as the condition $[K,X]_{X}=0$. Via the analysis in~\cite{Bourn-Gran}, this concept of central extension is also an instance of the notion coming from categorical Galois theory~\cite{Janelidze-Kelly}, namely the special case where one considers the Galois structure determined by the abelianisation functor.

\subsection{The Higgins commutator}\label{Higgins}
Central extensions may also be characterised in terms of the Higgins commutator~\cite{Higgins, Actions, MM-NC}, which is defined through a co-smash product. Given two objects $K$ and $L$ of $\X$, their \defn{co-smash product}~\cite{Smash}
\[
K\tensor L=\Ker\left(\left\langle\begin{smallmatrix}1_{K} & 0\\ 0 & 1_{L}\end{smallmatrix}\right\rangle\colon K+L\to K\times L\right)
\]
behaves as a kind of ``formal commutator'' of~$K$ and~$L$. In fact it is the Huq commutator of the two coproduct inclusions; see~\cite{Actions} and~\cite{MM-NC}. If $k\colon K \to X$ and $l\colon L \to X$ are subobjects of an object $X$, the \defn{Higgins commutator} $[K,L]\leq X$ is the subobject of $X$ given by the image of the induced composite morphism
\[
\xymatrix{K\tensor L \ar@{{ |>}->}[r]^-{\iota_{K,L}} \ar@{.{ >>}}[d] & K+L \ar[d]^-{\musgib{k}{l}}\\
[K,L] \ar@{{ >}.>}[r] & X.}
\]
If $K$ and $L$ are normal subobjects of $X$ and $K\join L=X$, then it turns out that the Higgins commutator $[K,L]$ is normal in $X$ and coincides with the Huq commutator. In particular, $[X,X]=[X,X]_{X}$. More generally, we always have $[K,X]=[K,X]_{X}$, so that a regular epimorphism $f\colon{X\to Y}$ is a central extension when either one of the three commutators $[K,X]=[K,X]_{X}=[K,X]^{\S}$ vanishes. In general the Huq commutator $[K,L]_{X}$ is the normal closure in $X$ of the Higgins commutator $[K,L]$. So, $[K,L]\leq [K,L]_{X}$ and $[K,L]=0$ if and only if $[K,L]_{X}=0$. An example in~\cite{AlanThesis} shows that in the category of non-associative rings the two commutators generally need not coincide. Thus the coincidence $[K,L]=[K,L]_{X}$ for all $K$, $L\normal X$ becomes a basic condition which a semi-abelian category may or may not satisfy; this condition, which we will denote by \defn{(NH)}, was introduced by Cigoli in his Ph.D. thesis~\cite{AlanThesis} and was studied further, by Cigoli together with the present authors, in~\cite{CGrayVdL1}.

\subsection{The ternary commutator}\label{ternary commutator}
The Higgins commutator does not preserve joins in general, but the defect may be measured precisely---it is a ternary commutator which can be computed by means of a ternary co-smash product. Let us extend the definition above: given a third subobject $m\colon M \to X$ of the object $X$, the \defn{ternary Higgins commutator} $[K,L,M]\leq X$ is the image of the composite
\[
\xymatrix@C=3em{K\tensor L\tensor M \ar@{{ |>}->}[r]^-{\iota_{K,L,M}} \ar@{.{ >>}}[d] & K+L+M \ar[d]^-{\muss{k}{l}{m}}\\ [K,L,M] \ar@{{ >}.>}[r] & X} 
\]
where $\iota_{K,L,M}$ is the kernel of
\[
\xymatrix@=6em{K+L+M \ar[r]^-{\left\langle\begin{smallmatrix} i_{K} & i_{K} & 0 \\
i_{L} & 0 & i_{L}\\
0 & i_{M} & i_{M}\end{smallmatrix}\right\rangle} & (K+L)\times (K+M) \times (L+M);}
\]
$i_k, i_L$ and $i_M$ denote the injection morphisms. The object $K\tensor L\tensor M$ is called the \defn{ternary co-smash product} of $K$,~$L$ and~$M$. Note that higher-order co-smash products and their associated commutators have been defined, but since we shall only use binary and ternary Higgins commutators, these will not be needed in this paper. Higgins commutators have good stability properties:

\begin{proposition}\cite{Actions,HVdL}\label{Proposition-Commutator-Rules}
For all $X_{1}$, $X_{2}$, $X_{3}\leq X$ and $M\leq X_{1}$ in $\X$, $n\in\{2,3\}$ and ${\sigma\in S_{n}}$ we have the following:
\begin{itemize}
\item Symmetry: $
[X_{\sigma^{-1}(1)},\dots,X_{\sigma^{-1}(n)}]= [X_{1},\dots,X_{n}]$.
\item Join decomposition: $[X_{1},X_{2}\join X_{3}]=[X_{1},X_{2}]\join [X_{1},X_{3}]\join [X_{1},X_{2},X_{3}]$.
\item Monotonicity: $[M,X_{2},\dots,X_{n}]\leq[X_{1}, \dots , X_{n}]$.
\item Removal of brackets: $[[X_{1},X_{2}],X_{3}]\leq [X_{1},X_{2}, X_{3}]$.
\item Removal of duplicates: if $X_{2}=X_{3}$ then $[X_{1},X_{2}, X_{3}]\leq [X_{1},X_{2}]$.\noproof
\end{itemize}
\end{proposition}

\subsection{Two lemmas}\label{Joins of subobjects in semi-abelian categories}
We end this preliminary section with two known lemmas
\begin{lemma}\label{Lemma Joins of subobjects in semi-abelian categories}
For each $K\normal X$ and $S \leq X$ in $\X$, the join $K\join S \leq X$ can be constructed as the preimage of $S/(K\meet S) \leq X/K$ along ${X \to X/K}$ as in the diagram
\[
\xymatrix@!0@=2.5em{
 K\meet S \ar@{{ |>}->}[rrr]\ar@{{ >}->}[ddd] &&& S \ar@{{ >}->}[ddd]|-{\hole} \ar@{-{ >>}}[rrr] \ar@{{ >}->}[dr] &&& S/(K\meet S) \ar@{{ >}->}[ddd]\\
 &&&& K\join S \sqpullback\ar@{-{ >>}}[urr]\ar@{{ >}->}[ddl] &\\
 &&&&\\
 K \ar@{{ |>}->}[rrr] \ar@{{ |>}->}[uurrrr] &&& X \ar@{-{ >>}}[rrr] &&& X/K.
}
\]
Moreover, when $S$ is normal in $X$, the join $K \join S$, being the preimage of the image of a normal subobject, is normal in $X$.\noproof
\end{lemma}
\begin{lemma}\cite[Lemma~2.6]{CGrayVdL1}\label{Lifting split extensions}
In a semi-abelian category, consider a point with chosen kernel as in bottom row of the diagram
\[
\xymatrix{0 \ar@{.>}[r] & K' \ar@{{ |>}->}[d]_-{\kappa} \ar@{{ |>}.>}[r] & X' \ar@{.>}[d] \ar@{.{ >>}}@<.5ex>[r] & Y \ar@{:}[d] \ar@{.>}@<.5ex>[l] \ar@{.>}[r] & 0\\
0 \ar[r] & K \ar@{{ |>}->}[r]_-{k} & X \ar@<.5ex>@{-{ >>}}[r]^-{f} & Y \ar@<.5ex>[l]^-{s} \ar[r] & 0}
\]
such that $k\comp \kappa$ is normal. Then this point lifts along $\kappa\colon K'\to K$ to yield a morphism of points with chosen kernels.\noproof
\end{lemma}

\section{Peri-abelian categories and the condition \WNH}\label{Section (WNH) and (PA)}

In this section we consider a weakening of the condition \NH\ from Subsection~\ref{Higgins}. Instead of requiring that Higgins commutators of pairs $K$, $L\normal X$ of normal subobjects are normal, we require this only in the special case where $K=L$. This condition is closely related to the concept of \emph{peri-abelian} category introduced in~\cite{Bourn-Peri}.

Following~\cite{CGrayVdL1}, we write $K\protosplit X$ if $K\normal X$ is the kernel of a split extension as in
\begin{equation}\label{split extension}\tag{$\star$}
\xymatrix{
0 \ar[r] & K \ar@{{ |>}->}[r]^-{k} & X \ar@<.5ex>@{-{ >>}}[r]^-{f} & Y \ar@<.5ex>[l]^-{s} \ar[r] & 0}
\end{equation}
and call $K$ a \defn{protosplit} normal subobject of $X$.

\begin{proposition}\label{Proposition Independence of surrounding object [K,K]}
For a semi-abelian category, the following are equivalent, and determine a condition which we shall denote by {\rm\bf (WNH)}:
\begin{enumerate}
\item if $K\normal X$ then $[K,K] \normal X$;
\item if $K \normal X$ then $[K,K]=[K,K]_X$;
\item if $K\normal X$ then $[K,K]_{K}\normal X$;
\item if $K\protosplit X$ then $[K,K]_{K}\normal X$;
\item each point with chosen kernel~\eqref{split extension} lifts to a morphism of points with chosen kernels
\[
\xymatrix{0 \ar[r] & [K,K]_K \ar@{{ |>}->}[d]_-{\mu_K} \ar@{{ |>}->}[r]^-{k'} & X' \ar@{->}[d]_-{x} \ar@{-{ >>}}@<.5ex>[r]^{f'} & Y \ar@{=}[d] \ar@{->}@<.5ex>[l]^{s'} \ar@{->}[r] & 0\\
0 \ar[r] & K \ar@{{ |>}->}[r]_-{k} & X \ar@<.5ex>@{-{ >>}}[r]^-{f} & Y \ar@<.5ex>[l]^-{s} \ar[r] & 0;}
\]
\item each point with chosen kernel~\eqref{split extension} induces a morphism of points with chosen kernels
\[
\xymatrix{
0 \ar[r] & K \ar@{{ |>}->}[r]^-{k} \ar@{-{ >>}}[d]_-{\eta_{K}=\coker \mu_{K}} & X \ar@<.5ex>@{-{ >>}}[r]^-{f} \ar@{->}[d]_-{q} & Y \ar@<.5ex>[l]^-{s} \ar[r] \ar@{=}[d] & 0\\
0 \ar[r] & K/[K,K]_K \ar@{{ |>}->}[r]_-{\widetilde k} & \widetilde X \ar@{-{ >>}}@<.5ex>[r]^{\widetilde f} & Y \ar@{->}@<.5ex>[l]^{\widetilde s} \ar@{->}[r] & 0;}
\]
\item each action on $K$ restricts to an action on $[K,K]_K$;
\item each action on $K$ induces an action on $K/[K,K]_K$.
\end{enumerate}
\end{proposition}
\begin{proof}
The equivalence of (i), (ii) and (iii) follows from the fact that the Huq commutator is always the normal closure of the Higgins commutator, and that the Higgins commutator of normal subobjects is normal in the join of those subobjects. (iv) is a special case of (iii). The fact that the conditions (v) and (vii) are equivalent and the conditions (vi) and (viii) are equivalent follows from the equivalence of categories between actions and points from \cite{BJK}.
The implication (vi)~\implies~(v) follows from the fact that the functor $\Ker\colon \Pt_{Y}(\X) \to \X$ preserves limits and in particular kernels. We will show that (iv)~\implies~(vi). Suppose that (iv) holds and that~\eqref{split extension} is a point with chosen kernel. From Lemma \ref{Lifting split extensions} we obtain the morphism of points as in (v) where $k \comp\mu_K$ is normal in $\X$ and so $x$ is normal in $\Pt_{Y}(\X)$. The induced morphism of points as in (vi) is obtained by taking the cokernel of $x$ in $\Pt_{Y}(\X)$. 
Finally, the last implication (vii)~\implies~(iii) follows from the fact that a subobject $S\leq X$ is normal in $X$ if and only if the conjugation on~$X$ restricts to $S$. Indeed, the conjugation action on $X$ restricts to an action on $K$ which by assumption restricts further to an action on $[K,K]_K$ as required.
\end{proof}

Note that for the category of groups, Condition (vii) of Proposition \ref{Proposition Independence of surrounding object [K,K]} corresponds to the fact that the commutator $[K,K]_K$ is a characteristic subgroup of~$K$: being invariant under automorphisms means precisely that every action on $K$ restricts to an action on $[K,K]_K$. This observation can be extended to arbitrary semi-abelian categories when the definition of characteristic subobject from \cite{CigoliMontoliCharSubobjects} is used; see \cite{CGrayVdL1} for the proof under the condition \NH.

\begin{proposition}\label{prop: arithmetical}
Every arithmetical category satisfies \WNH.
\end{proposition}
\begin{proof}
This is due to the fact that in an arithmetical category for each $K \normal X$ the commutator $[K,K]_X=K\meet K =K$, which follows essentially from \cite[Corollary~1.11.13]{Borceux-Bourn}.
\end{proof}
\begin{proposition}
For a semi-abelian category $\X$ the conditions \WNH\ and strong protomodularity are independent.
\end{proposition}
\begin{proof}
Let us consider the category $\X$ whose objects are sets equipped with the structure of an abelian group with operations denoted by $0$, $+$, $-$,  together with a binary operation $\cdot$ satisfying:
\begin{align*}
x+x=0, && x\cdot x =x, && x\cdot y=y\cdot x, && \text{and} && x\cdot 0=0.
\end{align*}
The morphisms in $\X$ are (as usual) the structure preserving maps. According to Proposition~2.9.2, Proposition~2.9.11 and Definition~2.9.13 in \cite{Borceux-Bourn}, to show that $\X$ is arithmetical it is sufficient to find a ternary operation $p$ which satisfies:
\begin{align*}
p(x,y,y)=x, && p(x,x,y)=y && \text{and} && p(x,y,x)=x.
\end{align*}
It is easy to check that $p(x,y,z) = x + y + x\cdot y + x\cdot z + y\cdot z$ has the desired properties, and hence by Proposition \ref{prop: arithmetical}, $\X$ satisfies \WNH. It remains only to show that $\X$ is not strongly protomodular. For that let $X$ be the boolean ring with two elements and consider the diagram
\[
\xymatrix{
X \ar[r]^-{\langle 1_{X},0\rangle}\ar[d]_{\langle 1_{X},0\rangle} & X\times X \ar@<0.5ex>[r]^-{\pi_2}\ar[d]_{u} & X \ar@<0.5ex>[l]^-{\langle 0,1_{X}\rangle} \ar@{=}[d]\\
X\times X \ar[r]_-{k} & C \ar@<0.5ex>[r]^-{p} & X\ar@<0.5ex>[l]^-{s}
}
\]
in $\X$ where $C$ is the object in $\X$ with underlying abelian group $X\times X\times X$ and with $\cdot$ defined by
\[
(x,y,z)\cdot (a,b,c) =
\begin{cases}
(xa,yb,zc) & \text{if $(x,y,z)=(a,b,c)$, $y=b=0$, $z=c=0$,}\\
&\text{$(x,y,z)=(0,0,0)$, or $(a,b,c)=(0,0,0)$;}\\
(1,1,zc) & \text{otherwise}
\end{cases}
\]
and $u$, $k$, $p$ and $s$ are defined by
\begin{align*}
u(x,y) = (x,0,y), && k(x,y) = (x,y,0), && p(x,y,z)=z && \text{and} && s(x) = (0,0,x),
\end{align*}
respectively.
It is easy to check that the above diagram is a morphism of points with chosen kernels where the morphism between the kernels is a normal monomorphism. To show that $u$ is not normal as a monomorphism of points it is sufficient to show that the composite $u\comp \langle 1_{X},0\rangle$ is not normal (as a monomorphism in $\X$) \cite{Gerstenhaber,Bourn-SP,Rodelo:Moore}. Taking into account that every normal monomorphism is the kernel of its cokernel, since $(1,0,0) \cdot (1,1,1) = (1,1,0)\notin u (\langle 1_{X},0\rangle(X))$ but the cokernel of $u\comp \langle 1_{X},0\rangle$ will send $(1,0,0)$ to $0$ and hence $(1,0,0)\cdot (1,1,1)$ to $0$, it follows that $u\comp \langle 1_{X},0\rangle$ is not normal. 

The fact that strong protomodularity does not imply \WNH\ follows from the fact that the category of non-associative rings is strongly protomodular but does not satisfy \WNH ~\cite[Example~5.3.7]{AlanThesis} \cite[Example~5.4]{CGrayVdL1}.
\end{proof}

\subsection{Peri-abelian categories}
A semi-abelian category $\X$ is said to be \defn{peri-abelian}~\cite{Bourn-Peri} if and only if for any $f\colon{X\to Y}$ in $\X$, the change of base functor
\[
f^{*}\colon{\Pt_{Y}(\X)\to \Pt_{X}(\X)}
\]
commutes with abelianisation. (The original definition of~\cite{Bourn-Peri} is actually given in a wider context, but we shall only consider the case of semi-abelian categories.)

\begin{proposition}\label{(SH) + (WNH) => (PA)}
For a semi-abelian category $\X$, the following are equivalent:
\begin{enumerate}
\item $\X$ is peri-abelian;
\item for any object $Y$ of $\X$, the diagram
\[
\xymatrix{\Pt_{Y}(\X) \ar[d]_-{\Ker} \ar[r]^-{\ab} & \Ab(\Pt_{Y}(\X)) \ar[d]^-{\Ker}\\
\X \ar[r]_-{\ab} & \Ab(\X)}
\]
commutes;
\item for any object $Y$ of $\X$, the square of left adjoint functors
\[
\xymatrix{\Pt_{Y}(\X) \ar[r]^-{\ab} & \Ab(\Pt_{Y}(\X)) \\
\X \ar[r]_-{\ab} \ar[u]|-{F=Y+(-)} & \Ab(\X) \ar[u]_-{G}}
\]
satisfies the Beck--Chevalley condition;
\item if $K\protosplit X$ then $[K,K]=[K,K]^{\S}$;
\item if $K\protosplit X$ then $[K,K,X]\leq[K,K]$;
\item if $K\normal X$ then $[K,K]=[K,K]^{\S}$;
\item if $K\normal X$ then $[K,K,X]\leq[K,K]$;
\item each action $\zeta\colon Y\flat K\to K$ induces an action $\theta\colon Y\flat A \to A$ and a morphism of $Y$-actions $\zeta\to \theta$, where $A=\ab(K)=K/[K,K]$ is the abelianisation of~$K$, such that
the diagram
\[
\xymatrix{
Y\flat (A\times A) \ar[r]^-{Y\flat m} \ar[d]_{\langle \theta \circ Y\flat \pi_1,\theta\circ Y\flat \pi_2\rangle}& Y\flat A \ar[d]^{\theta}\\
A\times A \ar[r]_-{m} & A, 
}
\]
in which $m$ is the multiplication of the group $A$, commutes.
\end{enumerate} 
Furthermore, \PA\ implies \WNH. When, in addition, $\X$ satisfies \SH\, then \WNH\ is equivalent to \PA. 
\end{proposition}
\begin{proof}
Condition (ii) is the special case of (i) where $f\colon{Y\to 0}$; it is explained in~\cite{Bourn-Peri} that this is sufficient. Condition (iii) is a reformulation of (ii).
Using the equivalence of categories between $\Pt_{Y}(\X)$ and $\X^{Y\flat -}$ from \cite{BJK}, it can be seen that (ii) and (viii) are equivalent. Indeed, the commutativity of the diagram in (viii) amounts to saying that $\theta$ is an abelian object; so if (ii) holds, then we can just take $\theta=\ab(\zeta)$. Conversely, the morphism $\zeta\to \theta$ provided by (viii) induces a morphism $\ab(\zeta)\to \theta$ which is both an isomorphism and a monomorphism, so that $A$ is indeed a kernel of $\ab(\zeta)$, and thus (ii) holds.

We now prove (ii)~\equiv~(iv). We recalled (in Subsection~\ref{Smith commutator} above) that a point $(f,s)\colon {X\rightleftarrows Y}$ is abelian if and only if the kernel pair of $f$ commutes with itself, and that the abelianisation of $(f,s)$ is obtained through the quotient $X/[K,K]^{\S}$. The kernel of this split extension is~$K/[K,K]^{\S}$. On the other hand, by definition, $\X$ is peri-abelian precisely when the kernel of the abelianisation is $K/[K,K]_{K}=K/[K,K]$, which happens if and only if $[K,K]=[K,K]^{\S}$. This proves the equivalence between (ii) and (iv).

By~\cite[Theorem~5.2]{HVdL}, $[K,K]^{\S}=[K,K]\join [K,K,X]$, which gives us (iv)~\equiv~(v) and (vi)~\equiv~(vii).
It is clear that (vi) implies (iv). We still have to prove (iv)~\implies~(vi), but here the proof of Theorem~2.3 in~\cite{MFVdL} may be repeated.

Since for $K \normal X$, the Smith commutator $[K,K]^{\S}$ is normal in $X$, it follows that condition (iv), which tells us that if $K \normal X$, then $[K,K]=[K,K]^{\S}$, gives Condition~(i) of Proposition~\ref{Proposition Independence of surrounding object [K,K]}. This proves that \PA\ implies~\WNH.

It remains to show that when the condition \SH\ holds we can also obtain the converse: \WNH\ implies \PA. Indeed, by Theorem~4.6 in~\cite{HVdL}, the inequality $[K,K,X]\leq [K,K]_{X}$ follows from \SH. But if~$[K,K]_{X}=[K,K]_{K}$ which follows from Condition (ii) of Proposition~\ref{Proposition Independence of surrounding object [K,K]}, this gives us condition (vii).
\end{proof}

\begin{corollary}\label{Categories of interest are peri-abelian}
All \emph{categories of interest} are peri-abelian, while categories of non-associative algebras, and the categories of loops and of digroups, are not. In particular, strong protomodularity~\cite{Bourn-SP} does not imply \PA.
\end{corollary}
\begin{proof}
Theorem~5.3.6 in Cigoli's thesis~\cite{AlanThesis} shows that any \emph{category of interest} satisfies \WNH, while those categories satisfy~\SH\ by action accessibility (\cite{BJ07} combined with~\cite{Montoli}). It is known that the condition \SH\ fails for the categories of loops and digroups~\cite{HVdL, Freese-McKenzie, Borceux-Bourn, Bourn2004}, and that the condition~\WNH\ fails for non-associative rings~\cite[Example~5.3.7]{AlanThesis} \cite[Example~5.4]{CGrayVdL1}. Note the that the latter category is strongly protomodular. 
\end{proof}

\section{The universal central extension condition}\label{Section (UCE)}
\begin{definition}\label{Condition-UCE}~\cite{CVdL}
We say that a semi-abelian category $\X$ satisfies the condition \defn{(UCE)} when: for each pair of composable central extensions $f\colon{A\to B}$ and $g\colon{B\to C}$, if $B$ is perfect, then the composite $g\comp f$ is a central extension.
\end{definition}

\begin{example}\label{Example NAA}
The variety of non-associative algebras over a field $\KK$ is semi-abelian, even strongly protomodular, but need not satisfy~\UCE~\cite{CVdL}. 
\end{example}

Our aim is to prove that $\X$ satisfies \UCE\ as soon as \PA\ holds. This implies in particular that \UCE\ holds for any category in which the Smith commutator and the Higgins commutator coincide. As a consequence of Corollary~\ref{Categories of interest are peri-abelian}, then all \emph{categories of interest}~\cite{Orzech} satisfy \UCE. Our argument is essentially a categorical version of the proof for groups given in~\cite{Milnor}.

\begin{lemma}\label{Lemma Join}
Let $K$ be the kernel of an extension $f\colon {A\to B}$ with a perfect codomain~$B$. Then $A=K\join [A,A]$.
\end{lemma}
\begin{proof}
It follows from Lemma \ref{Lemma Joins of subobjects in semi-abelian categories} that the join $K\join [A,A]$ is the preimage along~$f$ of the image of $[A,A]\normal A$ along $f$. Since the image of $[A,A]\normal A$ along $f$ is $[B,B]\normal B$, and $[B,B]=B$ because $B$ is perfect, it follows that $K\join [A,A] = A$ as required.
\end{proof}

\begin{lemma}\label{Lemma join of object in centre}
If $X$, $Y$, $Z \leq A$ are subobjects of $A$ such that $X \join Y = A$ and $[X,A]=0$, then $[A,Z]=[Y,Z]$.
\end{lemma}
\begin{proof}
Since by Proposition~\ref{Proposition-Commutator-Rules}, $[X,Y,Z] \leq [X,A,A] \leq [X,A]=0$, applying the same proposition, we see that
\begin{align*}
[A,Z] &= [X\join Y,Z]
 = [X,Z]\join [Y,Z] \join [X,Y,Z]
 = [Y,Z]. \qedhere
\end{align*}
\end{proof}

\begin{lemma}\label{Lemma Commutator Perfect}
If $f\colon {A\to B}$ is a central extension of a perfect object~$B$, then the composite $f\comp \mu_{A} \colon [A,A] \to B$ is a central extension with perfect domain.
\end{lemma}
\begin{proof}
It follows from Lemma~\ref{Lemma Join} and Lemma~\ref{Lemma join of object in centre} that $[A,A]$ is perfect. Since subobjects of central extensions are central extensions it follows that $f\comp \mu_A$ is a central extension. 
\end{proof}

The following lemma appeared in \cite{CVdL}. It also easily follows from Theorem 2.1 (1) $\Leftrightarrow$ (8) in \cite{Bourn-Gran}, which generalises Theorem 5.2 (i) $\Leftrightarrow$ (viii) in \cite{Janelidze-Kelly}; we repeat the proof to make the paper more self-contained.

\begin{lemma}\label{Lemma Morphisms into central}
Suppose that $f\colon{A\to B}$ is a central extensions and $P$ is a perfect object. If $p_{1}$, $p_{2}\colon P\to A$ are parallel morphisms such that $f\comp p_{1}=f\comp p_{2}$, then $p_{1}=p_{2}$.
\end{lemma}
\begin{proof}
Suppose $f\colon A\to B$ is a central extension, and $p_1$, $p_2\colon P \to A$ are parallel morphisms with perfect domain, such that $f\comp p_1=f\comp p_2$. Let $k\colon K\to A$ be the kernel of $f$, and let $\varphi \colon K\times A \to A$ be the morphism showing that $k$ and $1_A$ commute. It is well known that the kernel pair of $f$ can be presented as $(\varphi, \pi_2)\colon K\times A\rightrightarrows A$. Since $f\comp p_1=f\comp p_2$ it follows by the universal property of the kernel pair that there exists a morphism $q \colon P \to K\times A$ such that $\pi_2 \comp q = p_2$ and $\varphi \comp q = p_1$. Hence $q=\langle d,p_2 \rangle$ for some morphism $d\colon P\to K$. Since $P$ is perfect and $K$ is abelian it follows that $d$ is the zero morphism. We have $p_1 = \varphi \comp q = \varphi \comp \langle 0 ,p_2 \rangle = \varphi\comp \langle 0,1_{A}\rangle \comp p_2=1_A \comp p_2=p_2$, as required.
\end{proof}
For a composite of central extensions, using the above lemma and induction, we obtain the following lemma.
\begin{lemma}\label{Lemma Morphism into composite}
Suppose that $f\colon{A\to B}$ is a composite of central extensions and $P$ is a perfect object. If $p_{1}$, $p_{2}\colon P\to A$ are parallel morphisms such that $f\comp p_{1}=f\comp p_{2}$, then $p_{1}=p_{2}$.\noproof
\end{lemma}

\begin{proposition}\label{Proposition Bemol Perfect}
Let $\X$ be a semi-abelian category satisfying \PA. Consider objects $B$ and $P$ in $\X$. If $P$ is perfect then also $B\flat P$ is perfect.
\end{proposition}
\begin{proof}
Let us write $F\colon \X\to \Pt_{B}(\X)$ and $G\colon \Ab(\X)\to \Ab(\Pt_{B}(\X))$ for the left adjoints of the kernel functors as in Proposition~\ref{(SH) + (WNH) => (PA)}. Then 
\begin{align*}
\ab(B\flat P) &= \ab(\Ker(F(P)))
= \Ker(\ab(F(P)))\\
&= \Ker(G(\ab(P)))
= \Ker(G(0))
=0,
\end{align*}
so $B\flat P$ is perfect.
\end{proof}
\begin{remark}
It is well known that any non-abelian simple group is perfect (and trivially the corresponding statement is true in any semi-abelian category). It is natural to ask whether the above proposition holds for non-abelian simple objects. That is, is it true that for any objects $B$ and $X$ in a peri-abelian category, if $X$ is non-abelian and simple, then $B\flat X$ is non-abelian and simple? This turns out to be false even for the category of groups. In fact for any non-trivial groups $G$ and~$X$, the group $G\tensor X$ is a proper normal subobject of $G\flat X$: for non-trivial $g$ and $x$ in~$G$ and in $X$, respectively, the word $gxg^{-1}$ is in $G\flat X$ but not in $G\tensor X$, while the word $gxg^{-1}x^{-1}$ is in both. 
\end{remark}

\begin{lemma}\label{Lemma Kernel split epi unique}
Let $\X$ be a semi-abelian category and let $\kappa \colon K'\to K$ be a composite of central extensions with perfect domain. If $\X$ satisfies \PA, then
\begin{enumerate}

\item for each action $\zeta\colon B\flat K \to K$ there exists at most one action $\theta$ making the diagram
\[
\xymatrix{
B\flat K' \ar[r]^{B\flat \kappa } \ar[d]_{\theta} & B\flat K \ar[d]^{\zeta}\\
K'\ar[r]_{\kappa } & K
}
\]
commute;
\item
for each point with chosen kernel
\[
\xymatrix{
0 \ar[r] & K \ar@{{ |>}->}[r]_-{k} & A \ar@<.5ex>@{-{ >>}}[r]^-{f} & B \ar@<.5ex>[l]^-{s} \ar[r] & 0,}
\]
there exists, up to isomorphism, at most one lifting
\[
\xymatrix{0 \ar@{.>}[r] & K' \ar@{-{ >>}}[d]_{\kappa } \ar@{{ |>}.>}[r]^-{k'} & A' \ar@{.>}[d] \ar@{.{ >>}}@<.5ex>[r]^-{f'} & B \ar@{:}[d] \ar@{.>}@<.5ex>[l]^-{s'} \ar@{.>}[r] & 0\\
0 \ar[r] & K \ar@{{ |>}->}[r]_-{k} & A \ar@<.5ex>@{-{ >>}}[r]^-{f} & B \ar@<.5ex>[l]^-{s} \ar[r] & 0.}
\]
\end{enumerate}
\end{lemma}
\begin{proof}
Condition (ii) follows from (i) by the equivalence between actions and points. We prove that (i) follows from \PA.
Suppose $\theta\colon B\flat K'\to K'$ and $\phi\colon {B\flat K'\to K'}$ are two actions making the diagram in Condition (i)
commute. Since, by Proposition~\ref{Proposition Bemol Perfect}, the object $B\flat K'$ is perfect, and since $\kappa \comp \theta = \zeta \comp (B\flat \kappa ) = \kappa\comp \phi$ and $\kappa $ is a composite of central extensions, it follows by Lemma~\ref{Lemma Morphism into composite} that $\theta$ equals $\phi$.
\end{proof}

\begin{lemma}\label{Lemma (NH) => (UCE)}
Let $\X$ be a semi-abelian category satisfying~\PA\ and let $f\colon{A\to B}$ and $g\colon{B\to C}$ be central extensions in $\X$. If $B$ is perfect, then the morphism $\ker(g\comp f)\colon {\Ker(g\comp f) \to A}$ commutes with $\mu_A\colon{[A,A]\to A}$. 
\end{lemma}
\begin{proof}
Let $k\colon {K\to A}$ denote the kernel of $g\comp f$. Consider the diagram
\[
\xymatrix@C=3.5em{0 \ar@{.>}[r] & [A,A] \ar@{{ |>}->}[d]_-{\mu_{A}} \ar@{{ |>}.>}[r]^-{k'} & A' \ar@{.>}[d] \ar@{.{ >>}}@<.5ex>[r]^-{f'} & K \ar@{:}[d] \ar@{.>}@<.5ex>[l]^-{s'} \ar@{.>}[r] & 0\\
0 \ar@{->}[r] & A \ar@{-{ >>}}[d]_-{g\circ f} \ar@{{ |>}->}[r]^-{\langle 1_{A},0\rangle} & A\times K \ar@{->}[d]_-{(g\circ f)\times 1_{K}} \ar@{->}@<.5ex>@{-{ >>}}[r]^-{\pi_{2}} & K \ar@{=}[d] \ar@{->}@<.5ex>[l]^-{\langle k,1_{K}\rangle} \ar@{->}[r] & 0\\
0 \ar[r] & C \ar@{{ |>}->}[r]_-{\langle 1_{C},0\rangle} & C\times K \ar@<.5ex>@{-{ >>}}[r]^-{\pi_{2}} & K \ar@<.5ex>[l]^-{\langle 0,1_{K}\rangle} \ar[r] & 0}
\]
where we use Proposition~\ref{Proposition Independence of surrounding object [K,K]} (via Proposition~\ref{(SH) + (WNH) => (PA)}) to obtain the dotted lifting. The composite $f\comp \mu_{A}$ is a central extension with perfect domain by Lemma~\ref{Lemma Commutator Perfect}. As a consequence, the morphism $g\comp f\comp \mu_{A}\colon{[A,A]\to C}$ is a composite of central extensions with perfect domain. Lemma~\ref{Lemma Kernel split epi unique} now tells us that the lifting obtained above is unique up to isomorphism. However, the diagram
\[
\xymatrix@C=3.5em{0 \ar[r] & [A,A] \ar@{{ |>}->}[r]^-{\langle 1_{[A,A]},0\rangle}\ar[d]_{g\circ f\circ \mu_A} & [A,A]\times K \ar@<.5ex>@{-{ >>}}[r]^-{\pi_{2}}\ar[d]^{(g\circ f\circ \mu_A)\times 1_K} & K \ar@<.5ex>[l]^-{\langle 0,1_{K}\rangle} \ar[r]\ar@{=}[d] & 0\\
0 \ar[r] & C \ar@{{ |>}->}[r]_-{\langle 1_{C},0\rangle} & C\times K \ar@<.5ex>@{-{ >>}}[r]^-{\pi_{2}} & K \ar@<.5ex>[l]^-{\langle 0,1_{K}\rangle} \ar[r] & 0}
\]
is another lifting. Thus we obtain the diagram
\[
\xymatrix@C=3.5em{0 \ar[r] & [A,A] \ar@{{ |>}->}[d]_-{\mu_{A}} \ar@{{ |>}->}[r]^-{\langle 1_{[A,A]},0\rangle} & [A,A]\times K \ar@{->}[d]^-{\varphi} \ar@{->}@<.5ex>@{-{ >>}}[r]^-{\pi_{2}} & K \ar@{=}[d] \ar@{->}@<.5ex>[l]^-{\langle 0,1_{K}\rangle} \ar@{->}[r] & 0\\
0 \ar[r] & A \ar@{-{ >>}}[d]_-{g\circ f} \ar@{{ |>}->}[r]^-{\langle 1_{A},0\rangle} & A\times K \ar@{->}[d]_-{(g\circ f)\times 1_{K}} \ar@{->}@<.5ex>@{-{ >>}}[r]^-{\pi_{2}} & K \ar@{=}[d] \ar@{->}@<.5ex>[l]^-{\langle k,1_{K}\rangle} \ar@{->}[r] & 0\\
0 \ar[r] & C \ar@{{ |>}->}[r]_-{\langle 1_{C},0\rangle} & C\times K \ar@<.5ex>@{-{ >>}}[r]^-{\pi_{2}} & K \ar@<.5ex>[l]^-{\langle 0,1_{K}\rangle} \ar[r] & 0}
\]
which is a composite of morphisms of points with chosen kernels. It follows that the composite $\pi_1\comp \varphi$ makes the diagram
\[
\xymatrix@C=3.5em{[A,A] \ar[r]^-{\langle 1_{[A,A]},0\rangle} \ar@/_2ex/[rd]_-{\mu_{A}} & [A,A]\times K \ar[d]^{\pi_{1}\circ \varphi} & K \ar[l]_-{\langle 0,1_{K}\rangle} \ar@/^2ex/[ld]^-{k} \\
& A}
\]
commute, proving that $\mu_A$ and $k$ commute as required.
\end{proof}

\begin{theorem}\label{Theorem (NH) => (UCE)}
If a semi-abelian category satisfies~\PA, then it satisfies~\UCE. In particular, \SH\ $+$ \WNH~\implies~\UCE.
\end{theorem}
\begin{proof}
Let $B$ be a perfect object and let $f\colon{A\to B}$ and $g\colon{B\to C}$ be central extensions. It follows from Lemma~\ref{Lemma Join} that $A= [A,A] \join \Ker(f)$ and therefore from Lemma~\ref{Lemma join of object in centre} and Lemma~\ref{Lemma (NH) => (UCE)} that $[A,\Ker(g\comp f)] = [[A,A],\Ker(g\comp f)]=0$.
\end{proof}

\begin{corollary}\label{Categories of interest satisfy (UCE)}
All \emph{categories of interest} satisfy~\UCE. 
In particular, the category of groups, the category of rings, and the categories of associative, Lie or Liebniz algebras over a ring all satisfy~\UCE.\noproof
\end{corollary}

We include the following lemma to show that semi-abelian categories satisfying \UCE, together with other mild assumptions (valid in all \emph{categories of interest}, for instance) seem to provide the necessary tools to capture certain aspects of the theory of central extensions valid for groups which are not immediately given by categorical Galois theory.
For an object $X$ we will denote by $\Z(X) \leq X$ the largest subobject of $X$ such that $[X,\Z(X)]=0$. This subobject will be called the \defn{centre} of $X$. Note that the centre of an object is always a normal subobject.

\begin{lemma}[Gr\"un's Lemma]
Let $\X$ be a semi-abelian category admitting centres and satisfying \UCE---for instance, $\X$ may be any \emph{category of interest}. If $P$ is perfect, then the quotient $P/\Z(P)$ has trivial centre.
\end{lemma}
\begin{proof}
Consider the morphisms
\[
\xymatrix{P \ar[r]^-{p} & P/\Z(P) \ar[r]^-{q} & (P/(\Z(P))/(\Z(P/\Z(P)))}
\]
which are the quotients of $\Z(P)$ and $\Z(P/\Z(P))$ in $P$ and in $P/\Z(P)$, respectively. 
By taking kernels we obtain the exact sequence
\[
\xymatrix{
0 \ar[r]& \Z(P)=\Ker(p) \ar[r] & \Ker(q\comp p )\ar[r] & \Ker(q)\ar[r] & 0.
}
\]
Since $P$ is perfect and $p$ is a regular epimorphism it follows that $P/\Z(P)$ is perfect. Therefore, since $p$ and $q$ are central extensions, it follows from \UCE\ that $q\comp p $ is a central extension, and so $\Ker(q\comp p ) \leq \Z(P)$. It follows that $\Z(P)=\Ker(q\comp p )$ and so $\Ker(q) = \Z(P/\Z(P)) = 0$ as required. 
\end{proof}


\end{document}